\documentclass[12pt,reqno]{amsart}

\usepackage[shortlabels]{enumitem}
\usepackage{times}
\usepackage[T1]{fontenc}
\usepackage{mathrsfs}
\usepackage{latexsym}
\usepackage[dvips]{graphics}
\usepackage{epsfig}
\usepackage{amsmath,amsfonts,amsthm,amssymb,amscd, mathtools}
\input amssym.def
\input amssym.tex
\usepackage{color}
\usepackage{verbatim}
\usepackage{hyperref}
\usepackage{pgf,tikz}
\usetikzlibrary{arrows}

\usepackage{hyperref}

\usepackage{caption}
\usepackage{subcaption}
\usepackage{hyperref}

\usepackage[utf8]{inputenc}

\newcommand\be{\begin{equation}}
\newcommand\ee{\end{equation}}
\newcommand\bea{\begin{eqnarray}}
\newcommand\eea{\end{eqnarray}}
\newcommand\bi{\begin{itemize}}
\newcommand\ei{\end{itemize}}
\newcommand\ben{\begin{enumerate}}
\newcommand\een{\end{enumerate}}

\newtheorem{thm}{Theorem}[section]

\newtheorem{cor}[thm]{Corollary}
\newtheorem{lem}[thm]{Lemma}

\newtheorem{defi}[thm]{Definition}

\newtheorem{rek}[thm]{Remark}

\newcommand{\R}{\ensuremath{\mathbb{R}}}

\newcommand{\Oc}{\mathcal{O}}

\DeclarePairedDelimiter\abs{\lvert}{\rvert}     

\numberwithin{equation}{section}

\begin{document}
\definecolor{qqqqff}{rgb}{0.,0.,1.}
\definecolor{ffqqqq}{rgb}{1.,0.,0.}

\title{Characterizing optimal point sets determining one distinct triangle}
\author[H. N. Brenner]{Hazel N. Brenner}
\address{Department of Mathematicas, Virginia Tech University, Blacksburg, VA 24061}
\email{\textcolor{blue}{\href{mailto:hazelbrenner@vt.edu} {hazelbrenner@vt.edu}}}

\author[J. S. Depret-Guillaume]{James S. Depret-Guillaume}
\address{Department of Mathematics, Virginia Tech University, Blacksburg, VA 24061}
\email{\textcolor{blue}{\href{mailto:jdg@vt.edu} {jdg@vt.edu}}}

\author[E. A. Palsson]{Eyvindur A. Palsson}
\address{Department of Mathematics, Virginia Tech University, Blacksburg, VA 24061}
\email{\textcolor{blue}{\href{mailto:palsson@vt.edu} {palsson@vt.edu}}}

\author[R. W. Stuckey]{Robert W. Stuckey}
\address{Department of Mathematics, Virginia Tech University, Blacksburg, VA 24061}
\curraddr{Department of Mathematical Sciences, Kent State University, Kent, OH}
\email{\textcolor{blue}{\href{mailto:robert14@vt.edu} {robet14@vt.edu}},\textcolor{blue}{\href{mailto:rstucke1@kent.edu} {rstucke1@kent.edu}}}


\date{\today}
\thanks{The work of the third listed author was supported in part by Simons Foundation Grant \#360560.}

\begin{abstract}
In this paper we determine the maximum number of points in $\R^d$ which form exactly $t$ distinct triangles, where we restrict ourselves to the case of $t = 1$. We denote this quantity by $F_d(t)$. It was known from the work of Epstein et al. \cite{Epstein} that $F_2(1) = 4$. Here we show somewhat surprisingly that $F_3(1) = 4$ and $F_d(1) = d + 1$, whenever $d \geq 3$, and characterize the optimal point configurations. This is an extension of a variant of the distinct distance problem put forward by Erd\H{o}s and Fishburn \cite{ErdosFishburn}.

\end{abstract}

\maketitle
\tableofcontents

\section{Introduction}
In 1946 Erd\H{o}s proposed his distinct distance conjecture, which stated that  any set of $n$ points in the plain will define at least $\Omega(n/\sqrt{\log n})$ distinct distances \cite{Er46}. Since that time, optimal points sets have been a heavily studied topic within the field of discrete geometry. Guth and Katz made significant progress towards proving this conjecture when they showed in 2010 that a set of $n$ points in the plane defined at least $\Omega(n/\log n)$ distinct distances. The analogous problems in dimensions 3 and higher remain open.

In 1996 Erd\H{o}s and Fishburn asked a question related to this: Given a positive integer $k$, what is the maximum number of points which can be embedded in the plane such that precisely $k$ distinct distances are defined, and can all such point configurations be characterized? In their paper, Erd\H{o}s and Fishburn characterized the optimal configurations for $1 \leq k \leq 4$, and this was extended by Shinohara for $k = 5$ and by Wei for $k = 6$. Further, Erd\H{o}s conjectured that for sufficiently large values of $k$, an optimal point configuration exists in the triangular lattice, which continues as an open conjecture. (Figure ~\ref{fig: distances} shows the optimal configurations for $k$ distinct distances in the plane when $2\leq k \leq 6$.)

\begin{center}
\begin{figure}[h]
\includegraphics[scale=0.8]{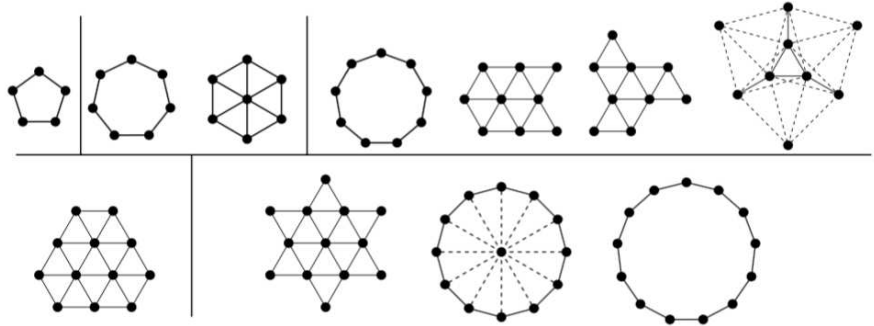}
\caption{Optimal, or maximal, point set configurations determining exactly $k$ distinct distances in the plane, for $2 \leq k \leq 6$ \cite{BMP}.  For all $k$, $2<k\leq 6$, there exists a configuration in the triangular lattice; Erd\H{o}s conjectured that this is always true when $k$ is sufficiently large.}
\label{fig: distances}
\end{figure}
\end{center}

Erd\H{o}s' distance problem can be extended to consider triangles in place of distances. Since the set of distances generated by a point set can be thought of as being determined by the collection of 2-point subsets, we may analogously consider the set of triangles formed by a point set as determined by the collection of 3-point subsets. Erd\H{o}s' distance problem then becomes: What is the minimum number of distinct triangles formed by a collection of $n$ points in the plane? Hence, the analogue of Erd\H{o}s' and Fishburn's problem is: Given $t$ distinct triangles, with $t$ fixed, what is the maximum number of points, $n$, placed in the plane which define exactly $t$ distinct triangles? Epstein et al. focused on the latter of these analogues and showed that $n = 4$ for $t = 1$, and $n = 5$ for $t = 2$ \cite{Epstein}. Maximal point sets in the plane remains an open question for higher values of $t$. As mentioned above, the higher dimensional analogues of Erd\H{o}s' distance problem are as yet open. Here we concern ourselves with the higher dimensional analogue of Erd\H{o}s' and Fishburn's question, rather than with higher values of $k$. Our main result is the following: 
\begin{thm}
  \label{thm: Main}
	Suppose $S\subset \R^d$ determines a single distinct triangle $T$.
	\begin{enumerate}
		\item \label{thmpart:equil} If $T$ is equilateral, then $S$ is contained in the set of vertices of a regular $d$-simplex, and in particular $\abs{S} \leq d+1$.
		\item \label{thmpart:nonequil} If $T$ is not equilateral, then $\abs{S}\leq 4$. 
	\end{enumerate}
\end{thm}

We can then state the following corollary in the language of optimal point configurations:

\begin{cor}
Let $F_d(t)$ denote the maximum number of points which can be placed in $\R^d$ to determine exactly $t$ distinct triangles. Then
\begin{enumerate}
    \item \label{thmpart: dim3}
    $F_3(1) = 4$ and the only configurations which achieve this are the vertices of a square, a rectangle, or a tetrahedron, and
    \item \label{thmpart: dimd}
    $F_d(1) = d + 1$ when $d > 3$ and the only configuration which achieves this is the regular $d$-simplex.
\end{enumerate}
\end{cor}

We also make three observations. First, in $\R^d$, $d > 3$ the $d$-simplex is the unique optimal point configuration, yet in dimensions 2 and 3 this is not so. Second, in addition to the above, the $2$-simplex fails to be optimal in $\R^2$. Third, note that in $\R^3$ there are two optimal configurations, viz. the tetrahedron (3-simplex) and the rectangle (to include the square), while in $\R^2$ and $\R^d$, $d > 3$, there is a single family of solutions (by considering the square to be a special case of the rectangle.), and hence, the $d$-simplex fails to be unique. This transition that happens in $\R^3$ is surprising and novel. In addition to the above, we have the following notable remark:

\begin{rek} \label{rek:weird_tetrahedra}
	In $d=3$, both Theorem \ref{thm: Main} (\ref{thmpart:equil}) and (\ref{thmpart:nonequil}) yield optimal configurations. These configurations are specifically the vertices of the regular tetrahedron (\ref{thmpart:equil}) and the vertices of the square, the vertices of a tetrahedron with isosceles faces and the vertices of a tetrahedron with scalene faces (\ref{thmpart:nonequil}). These can be uniquely determined as distance graphs which can be realized in $\R^3$ in the above ways. 
\end{rek}

In Section \ref{sec:rekconstructs} we offer a proof of this remark by way of a construction of said distance graphs and point sets in $\R^3$ that satisfy them. This remark is particularly interesting as our framework for an upcoming paper in preprint arrives at constructions for optimal configurations determining few distinct triangles by considering the number of distinct distances that can be determined by such configurations \cite{Brenner}. This remark shows that there can exist distinct optimal configurations determining a given number of distinct triangles that determine different numbers of distinct distances. Interestingly, this is not the case for optimal configurations determining two distinct triangles, which may only determine two distinct distances.

\section{Definitions and Lemmas}
We formalize the concepts of a triangle and set out our notation with the following definitions:

\begin{defi}
Given a finite point set $P \subset \R^d$, $d \geq 3$, we say two triples $(a,b,c)$, $(a',b',c') \in P^3$ are equivalent if there is an isometry mapping one to the other, and we denote this as $(a,b,c) \sim (a',b',c')$.
\end{defi}

\begin{defi}
Given a finite point set $P \subset \R^d$, $d \geq 3$, we denote by $P_{nc}^3$ the set of non-collinear triples $(a,b,c) \in P^3$.
\end{defi}

\begin{defi}
Given a finite point set $P \subset \R^2$, we define the set of distinct triangles determined by $P$ as
\begin{equation}
T(P) := P_{nc}^3 / \sim.
\end{equation}
\end{defi}

In this paper when we discuss and count the number of distinct triangles of a finite point set $P \in \R^d$ we are precisely working with the set $T(P)$. Note that this excludes degenerate triangles where all three points lie on a line.

\begin{defi}
\label{defi: Distance}
Let $p$ and $q$ be points in $\R^d$ for $d \geq 1$. We denote the Euclidean distance between $p$ and $q$ by $d(p,q)$.
\end{defi}

Theorem \ref{thm: Main} (\ref{thmpart:equil}) is a direct consequence of the following lemma:
\begin{lem}
\label{lem: Simplex}
Let $S$ be a set of points in $\R^d$, $d\geq 3$, which defines a single distinct equilateral triangle. Then $S$ has at most $d + 1$ points.
\end{lem}

\section{Proof of Theorem \ref{thm: Main}}

As stated previously, Theorem \ref{thm: Main} (\ref{thmpart:equil}) follows directly from Lemma \ref{lem: Simplex}, so we will omit proof in this section in favor of proving Lemma \ref{lem: Simplex} in Section \ref{sec:pfsoflemmas}. Then, to prove Theorem \ref{thm: Main} (\ref{thmpart:nonequil}), we will consider the case where $T$ is isosceles and the case where $T$ is scalene separately. In both cases, we assume towards a contradiction that there exists a point set $S$ containing 5 points determining such a triangle. Our argument will be made purely on distance graphs and will thus not depend on dimension.

\begin{proof}[Proof of \ref{thm: Main} (\ref{thmpart:nonequil})]
  Assume towards a contradiction that there exists a point set $S$ containing five points which determines one distinct non-equilateral triangle. For convenience, we then split into cases, dealing with scalene and isosceles triangles separately. 
	\begin{description}[align=left]
		\item[Scalene] Fix an arbitrary point $\Oc$ in $S$ and consider the distances from $\Oc$ to the remaining 4 points. Note that clearly a point set determining exactly one distinct, scalene triangle determines only three distinct distances. So, by the pigeonhole principle, two of the distances from $\Oc$ to the remaining points are equal. Without loss of generality, say that the repeated distance is $d_1$ and specifically $\Oc A = \Oc B = d_1$. Then, clearly $\triangle \Oc AB$ is an isosceles triangle, but we assumed that the only distinct triangle determined by this point set is scalene, so we have the desired contradiction.
		\item[Isosceles] Let $d_1$ denote the repeated edge length of $T$ and $d_2$ the remaining edge length. Similarly to the above, fix an arbitrary point $\Oc$ and consider the distances from $\Oc$ to the remaining points. Clearly, by the same argument as the above, if any two of these distances were $d_2$, an isosceles triangle with repeated edge length $d_2$ would be determined, which would be a contradiction. So, assume that at least three of the four distances are $d_1$ and label the three points determining them $A$, $B$ and $C$. Then consider each of the triangles consisting of two of the points and $\Oc$, e.g. $\triangle \Oc AB$. Clearly this must be congruent to $T$, and $\Oc A = \Oc B = d_1$, so we must have $AB=d_2$. The same holds for all such triangles, so we have $AB = BC = AC = d_2$. But, again, we assumed that the only triangle determined by $S$ was isosceles, so this equilateral triangle of edge length $d_2$ yields a contradiction.  
	\end{description}

\end{proof}

\section{Proofs of Lemmas}\label{sec:pfsoflemmas}
It is well known in the literature that Lemma \ref{lem: Simplex} holds, and that in $\R^d, d\geq 2$ a set of points determining a single distinct distance has at most $d + 1$ points. However, in the interest of completeness, we include here a proof of the lemma by induction on the dimension $d$:
\begin{proof} [Proof of Lemma \ref{lem: Simplex}]~
\begin{description}
    \item[Base Case ($d = 3$)] Let $S = \{A_1, A_2, A_3, A_4, A_5\} \subset \R^d = \R^3$ be a point set containing $d + 2 = 5$ points, which defines a single distinct equilateral triangle, call it $T$. Thus, the triangle $\triangle A_1A_2A_3$ must form the triangle $T$. Define $e$ to be the edge length of $T$, and let $P$ denote the plane defined by $\{A_1, A_2,A_3\}$.
    
    Since $S$ defines an equilateral triangle, it follows that $A_4$ and $A_5$ must be equidistant from $\{A_1,A_2,A_3\}$ and lie upon a line normal to $P$, which goes through a point $p \in P$, where $p$ is equidistant to $\{A_1,A_2,A_3\}$, $p$ is called the circumcenter of the equilateral triangle $\triangle A_1A_2A_3$.
    
    Since $p$ is the circumcenter of the equilateral triangle $\triangle A_1A_2A_3$, it follows that
    $$d(A_1,p) = d(A_2,p) = d(A_3,p) = \dfrac{\sqrt{3}}{3}e$$
    Since $S$ defines only the triangle $T$, it follows that $d(A_4,A_5) = e$. Since $A_4$, $A_5$, and $p$ lie upon the same line, it follows that:
    $$ d(A_4, p) + d(p, A_5) = d(A_4, A_5)$$
    Since $A_4$ and $A_5$ are equidistant from $\{A_1,A_2,A_3\}$, they are also equidistant from the plane $P$, and hence, $d(A_4, p) = d(p, A_5) = \frac{e}{2}$. Applying the Pythagorean Theorem we obtain:
    $$d(A_4, p)^2 + d(A_1, p)^2 = d(A_4, A_1)^2$$
    Which yields:
    $$\left(\frac{e}{2}\right)^2 + \left( \frac{\sqrt{3}e}{3} \right)^2 = (e)^2$$
    Thus, we obtain:
    $$\frac{1}{4}e^2 + \frac{1}{3}e^2 = e^2$$
    Which implies that $\frac{7}{12} = 1$, a clear contradiction. We note that the vertices of a $3$-simplex, the regular tetrahedron, give a configuration in $\R^3$ which defines a single distinct equilateral triangle and has $3 + 1 = 4$ points. Therefore, a set $S$ defining a single distinct equilateral triangle in $\R^3$ can have at most $3 + 1 = 4$ points.
    
    \item[Inductive Assumption] Suppose that for dimensions $n < d$, a point set $S$ defining a single distinct equilateral triangle can have at most $n+1$ points.
    
    \item[Inductive Step] Now let $S = \{A_1, \ldots, A_d, A_{d+1}, A_{d+2}\} \subset \R^d$ be a point set containing $d + 2$ points, which defines a single distinct equilateral triangle, call it $T$. Thus, the points $\{A_1, \ldots, A_d\}$ must be such that any triplet forms the triangle $T$, and so they must form a $(d-1)$-simplex (This by our Inductive Assumption). Call this $(d-1)$-simplex $\triangle A_1\cdots A_d$. Let $e$ be the edge length of $T$, and let $P$ be the $d-1$ dimensional hyper-plane defined by $\{A_1, \ldots, A_d\}$.
    
    Since $S$ defines an equilateral triangle, it follows that $A_{d+1}$ and $A_{d+2}$ must be equidistant to $\{A_1, \ldots, A_d\}$ and lie upon a line normal to $P$, which passes through a point $p \in P$ (vis. the circumcenter of $\triangle A_1 \cdots A_d$), where $p$ is equidistant from $\{A_1, \ldots, A_d\}$.
    
    From this point, we can follow the same construction as in the Base Case, using the points $A_1$, $p$, $A_{d+1}$, and $A_{d+2}$ to arrive at a contradiction. We note here that the $d$-simplex gives a configuration in $\R^d$ which defines a single distinct equilateral triangle and has $d+1$ points. Thus, a set $S$ in $\R^d$ defining a single distinct equilateral triangle can have at most $d+1$ points, as desired.
\end{description}
\end{proof}

\section{Constructions for Remark \ref{rek:weird_tetrahedra}}\label{sec:rekconstructs}

Clearly the regular tetrahedron is an optimal configuration (of four points) determining one distinct triangle in three dimensions, and it is the unique configuration satisfying Theorem \ref{thm: Main} (\ref{thmpart:equil}). Then, in the following, we will characterize the distances graphs of the four point (and thus optimal in three dimensions) sets satisfying Theorem \ref{thm: Main} (\ref{thmpart:nonequil}). We then provide constructions of point sets in $\R^3$ satisfying these distance graphs. For convenience, we will again consider the isosceles and scalene cases separately. Assume in all that follows that $S$ is a set of four points determining one distinct triangle of the respective geometry.

\begin{description}
	\item[Isosceles] Let $d_1$ be the repeated edge length of $T$. Following the framework used in the main proof, fix a point $\Oc$ in $S$ such that $\Oc A = d_1$ and $\Oc B = d_2$ for some $A$ and $B$ in $S$. Then, notice that given the remaining point $C$ in $S$, $\triangle \Oc AB \simeq \triangle \Oc CB \simeq T$, so specifically $\Oc C = BC = AB = d_1$, and similarly, we must have $AC = d_2$. 
	
	Notice that any four non-coplanar points form the vertices of a tetrahedron (if the points are coplanar, this distance graph is clearly uniquely realized by the vertices of the square). And, by the above, a tetrahedron $ABCD$ satisfying the above distance graph must have $AB = d_2$ and $CD = d_2$ and the remaining edges $d_1$. To construct such a tetrahedron, consider taking points $P= (d_2/2, 0, 0)$, $Q = (-d_2/2, 0, 0)$, $R=(0, d_2/2, 0)$ and $S=(0, -d_2/2, 0)$. Note that for $d_1 = \sqrt{2}/2 d_2$, $PQRS$ satisfies the above distance graph, although it is planar. Then, by arbitrarily translating $R$ and $S$ by the same distance along the $z$ axis, we can construct any tetrahedron satisfying this distance graph (clearly no tetrahedron with $d_1$ not satisfying this inequality may exist). 
	
	\item[Scalene] Let $A$, $B$ and $C$ determine $\triangle ABC \simeq T$ in $S$. Without loss of generality let $AB = d_1$, $BC = d_2$ and $AC = d_3$. Then, notice that each point may determine each distance exactly once (otherwise they would determine an isosceles triangle, producing a contradiction similar to that used in the main proof. Then, since each point $A$, $B$ and $C$ determines exactly two of the distances, we may simply fill in that $AD= d_2$, $BD = d_3$ and $CD = d_1$. It is easy to verify that this distance graph determines only one distinct triangle.
	
	Similarly to the isosceles case, we observe that if all four points are coplanar, this distance graph uniquely determines a rectangle. Supposing instead that four points $A$, $B$, $C$ and $D$ satisfy the above distance graph and are non-coplanar (thus form a tetrahedron), we clearly must then have that, up to relabeling, $AB = CD = d_1$, $AC = BD = d_2$ and $AD = BC = d_3$. Namely, ``opposite'' pairs of edges of the tetrahedron are congruent. To construct all such tetrahedra, consider fixing points $P = (d_1/2, 0, 0)$ and $Q = (-d_1/2, 0, 0)$. Then, fix $R^\prime = (0, d_1/2, z^\prime)$ and $S\prime = (0, -d_1/2, z^\prime)$ for arbitrary $z^\prime$. Consider the unique circle lying in the plane $z = z^\prime$ passing through $R^\prime$ and $S^\prime$. Then choose $R$ and $S$ as the endpoints of any diameter of this circle such that $R, S \neq R^\prime, S^\prime$ and $P$, $Q$, $R$ and $S$ are not all coplanar. Clearly, then, PQRS satisfies the above distance graph.
\end{description}

\section{Acknowledgements}

We would like to thank an anonymous referee for suggesting the inclusion of Theorem \ref{thm: Main} in its current form. Their proposed restructuring around this result has significantly improved the paper.


\end{document}